\documentclass[10pt,reqno]{amsart} 

\usepackage{amssymb,latexsym}
\usepackage{cite} 

\usepackage[height=190mm,width=130mm]{geometry} 

\usepackage{amsfonts}
\usepackage{amsmath}
\usepackage{amscd}
\usepackage{amssymb}
\usepackage{amsthm}
\usepackage{bbm}
\usepackage{bm}
\usepackage{enumerate}
\usepackage{enumitem}
\usepackage{hyperref}
\usepackage{latexsym}
\usepackage{mathabx}
\usepackage{mathdots}
\usepackage{mathtools}

\usepackage{amsmath}
\usepackage{amssymb}
\usepackage{amsthm}
\usepackage{bbm}
\usepackage{enumerate}
\usepackage{latexsym}
\usepackage{mathdots}
\usepackage{geometry}

\usepackage{hyperref}   

\newcommand{\bbar}{\overline{b} }

\newcommand{\C}{\mathbb{C}}
\newcommand{\calC}{\mathcal{C}}

\newcommand{\Chat}{\widehat{\C}}
\newcommand{\D}{\mathcal{D}}
\newcommand{\del}{\partial}

\newcommand{\E}{\mathcal{E}}

\newcommand{\F}{\mathcal{F}}
\newcommand{\I}{\mathcal{I}}

\newcommand{\Ip}{\I_{+}}

\newcommand{\Omo}{\Omega_{0}}
\newcommand{\Sg}{\mathcal{S}_{g}}

\newcommand{\Szero}{\mathcal{S}_{0}}
\newcommand{\SL}{\textup{SL}}

\newcommand{\wt}{\textup{wt}}
\newcommand{\Z}{\mathbb{Z}}

\newcommand{\omt}{\widetilde{\omega}}

\theoremstyle{plain}

\newtheorem{lemma}{Lemma}

\newtheorem{proposition}{Proposition}

\theoremstyle{definition}
\newtheorem{definition}{Definition}

\theoremstyle{remark}
\newtheorem{remark}{Remark}

\numberwithin{equation}{section} 

\newcommand{\g}{\ensuremath{\Gamma}}

\newcommand{\ps}{{\raise 1pt\hbox{\tiny (}}}

\newcommand{\pss}{{\raise 1pt\hbox{\tiny [}}}
\newcommand{\pdd}{{\raise 1pt\hbox{\tiny ]}}}
\newcommand{\pd}{{\raise 1pt\hbox{\tiny )}}}

\newcommand{\bs}{{\raise 1pt\hbox{\tiny [}}}
\newcommand{\bd}{{\raise 1pt\hbox{\tiny ]}}}

\def\cross{\mathinner{\mathrel{\raise0.8pt\hbox{$\scriptstyle>$}}
                 \joinrel\mathrel\triangleleft}}

\usepackage{citehack}
\usepackage{amsmath,amsthm}
\usepackage{amsfonts}
\usepackage{amssymb}
\usepackage{longtable}
\usepackage[matrix,arrow,curve]{xy}

\usepackage{lipsum}

\def\K{\mathcal{K}}

\usepackage{stackrel}

\newcommand{\be}{\begin{equation}}
\newcommand{\ee}{\end{equation}}

\newcommand{\nn}{\nonumber \\}

\newcommand{\one}{\mathbf{1}}

\newcommand{\nc}{\newcommand}
\nc{\cali}{\mathcal}
\nc{\on}{\operatorname}
\nc{\Wick}{{\mb :}}

\nc{\ddz}{\frac{\partial}{\partial z}}
\nc{\ch}{\mbox{ch}}
\nc{\Oo}{{\cali O}}
\nc{\cond}{|\,}
\nc{\bib}{\bibitem}
\nc{\pone}{\Pro^1}
\nc{\pa}{\partial}
\nc{\arr}{\rightarrow}
\nc{\larr}{\longrightarrow}
\nc{\ket}{\rangle}
\nc{\bra}{\langle}
\nc{\gam}{\bar{\gamma}}
\nc{\q}{\widetilde{Q}}
\nc{\ep}{\epsilon}
\nc{\su}{\widehat{{\mf s}{\mf l}}_2}
\nc{\sw}{{\mf s}{\mf l}}
\nc{\h}{{\mf h}}
\nc{\n}{{\mf n}}
\nc{\ab}{\mf{a}}
\nc{\is}{{\mb i}}
\nc{\js}{{\mb j}}
\nc{\bi}{\bibitem}
\nc{\He}{{\cali H}}
\nc{\inv}{^{-1}}
\nc{\ol}{\overline}
\nc{\wh}{\widehat}
\nc{\dst}{\displaystyle}

\nc{\delt}{\partial_t}
\nc{\ddt}{\frac{\partial}{\partial t}}
\nc{\delx}{\partial_x}
\nc{\mb}{\mathbf}
\nc{\mf}{\mathfrak}

\nc{\mbb}{\mathbb}
\nc{\Ctt}{\C((t))}
\nc{\Ct}{\C[t,t\inv]}

\nc{\ghat}{\wh{\g}}

\nc{\un}{\underline}
\nc{\mc}{\mathcal}
\nc{\BB}{{\mc B}}
\nc{\bb}{{\mf b}}
\nc{\kk}{{\mf k}}
\nc{\frob}{\times}
\nc{\sm}{\setminus}
\nc{\Pp}{{\mathbb P}^1}
\nc{\Aa}{{\mc A}}

\nc{\AutO}{\on{Aut}\Oo}
\nc{\AUTO}{\un{\on{Aut}}\Oo}
\nc{\AUTK}{\un{\on{Aut}}\K}
\nc{\Heout}{\He_{\out}}
\nc{\Hetil}{{\widetilde\He}}
\nc{\wb}{\overline}

\nc{\Res}{\on{Res}}
\nc{\pitil}{\Pi}
\nc{\Ctil}{\wt{C}}
\nc{\auto}{\on{Aut} \Oo}
\nc{\phitil}{\wt{\phi}}
\nc{\gz}{\g_{\vec z}}
\nc{\tensorM}{\bigotimes_{i=1}^N{\mathbb M}_i}
\nc{\tensorW}{\bigotimes_{i=1}^N W_{\nu_i,k}}
\nc{\out}{\on{out}}

\nc{\m}{{\mathfrak m}}

\nc{\gx}{\g^0_{\vec x}}

\nc{\hx}{\He^0_{\vec x}}
\nc{\tensorpi}{\pi_{\nu_1,\ldots,\nu_N}^\kappa}
\nc{\Phizw}{\Phi_{\vec w}({\vec z})}
\nc{\Pro}{{\mathbb P}}

\nc{\De}{\Delta}

\nc{\us}{\underset}

\nc{\Ll}{\mc L}
\nc{\dR}{\on{dR}}

\nc{\T}{{\mc T}}

\nc{\Xn}{\overset{\circ}X{}^n} \nc{\Dn}{\overset{\circ}D{}^n}
\nc{\Dxn}{\overset{\circ}D{}^n_x} \nc{\varphitil}{\wt{\varphi}}

\nc{\lf}{{\mf l}}
\nc{\GL}{{}^L G}
\nc{\Vir}{\on{Vir}}

\begin{document}
\title[Schottky vertex operator cluster algebras] 
{Schottky vertex operator cluster algebras}   
\author{A. Zuevsky}
\address{Institute of Mathematics \\ Czech Academy of Sciences\\ Zitna 25, 11567 \\ Prague\\ Czech Republic}

\email{zuevsky@yahoo.com}







\begin{abstract}
Using 
recursion formulas for vertex operator algebra higher genus characters  
with formal parameters identified with   
local coordinates around marked points on a Riemann surface of arbitrary genus, 
 we introduce the notion of a vertex operator cluster algebra structure. 
Cluster elements and mutation rules are explicitly defined, and the simplest example of a 
vertex operator cluster algebra is presented. 
\end{abstract}

\keywords{Cluster algebras; vertex operator algebras; Riemann surfaces; vertex algebra characters; Schottky uniformization}

\maketitle
\section{Introduction} 
The theory of cluster algebras is connected to many different areas of mathematics, e.g., 
the representation theory of finite dimensional algebras, Lie theory, Poisson geometry and
Teichm\"uller theory\cite{GSV1, GSV2, GSV3}.  
 Among these topics are 
 dilogarithm identities for conformal field theories\cite{N1, N2},  
quantum algebras
\cite{GLS, HL}, quivers \cite{Ke1, Ke2, Na}. 
Cluster algebras 
have 
numerous applications 
\cite{FG1, FG2, FG3, FG4, GSV1, GSV2, FST, KS,   
N1, N2, DFK}. 
Cluster algebras appear in several applications in Conformal Field Theory 
\cite{DFK, N1, N2}.  
It is natural to find an analogue of a cluster algebra structure in the language of vertex operator algebras. 
In particular, a structure that incorporates non-commutative nature of vertex operator algebra relations. 

In this paper we make use of the higher genus reduction formulas for 
vertex algebra characters. 
The 
recursion properties of vertex operator algebra characters  allow us to 
  introduce 
an algebraic structure which we call a vertex operator cluster algebra,  simultaneously 
incorporating  properties of cluster algebras, non-commutative nature, 
and analytical and geometrical features of character function theory of vertex operator algebras on Riemann surfaces. 
Vertex operator cluster algebra seeds are defined over non-commutative variables 
(elements of vertex algebra), coordinates around 
marked points on a Riemann surface, 
and functions depending on 
%
a number of vertex operators. 

In Section \ref{definition} we recall basic definitions related to cluster algebras. 
In Section \ref{corfu} the construction of $n$-point characters 
on the  Schottky reparameterization of a genus $g$ 
 Riemann surface 
is reminded.
A short introduction to vertex operator algebras is given in Appendix \ref{vertex}. 
Appendix \ref{cluster} contains the formulation of a vertex operator cluster algebra structure and Proposition \ref{pizda} describing 
the involutivity property of a vertex algebra setup. 
Appendix \ref{sec:Genusg} contains a description of the Schottky parameterization of a genus $g$ Riemann surface. 
Appendix \ref{derivation} describes auxiliary objects and matrices needed for the Schottky genus $g$ Zhu reduction 
formula. 
 
\section{Vertex algebra characters and Zhu reduction formula in Schottky parameterization} 
\label{corfu}
In this section we recall \cite{TW} the construction and Proposition \ref{theor:ZhuGenusg} 
 for vertex operator algebra 
character functions on a genus $g$ Riemann surface. 
In particular, 
the  
formal partition and $n$-point correlation functions for a vertex operator algebra associated to a genus $g$ Riemann surface $\Sg$ are introduced in the Schottky  scheme  with sewing relation \eqref{eq:SchottkySewing2}.  
All expressions here are functions of formal variables $w_{\pm a}$, $\rho_{a}$ 
 and vertex operator parameters. 
Then we recall 
 the genus $g$ Zhu recursion formula with universal  coefficients that have a geometrical meaning and 
are meromorphic on a Riemann surface 
$\Sg$ for all $(w_{\pm a},\rho_{a})\in\mathfrak{C}_{g}$ (for notations see Appendix \ref{sec:Genusg}).  
 These coefficients are 
generalizations of the 
elliptic Weierstrass functions~\cite{Z}. 

For a $2g$ vertex algebra $V$ states 
\[
\bm{b}=(b_{-1}, b_{1}; \ldots;  b_{-g} ;b_{g}),  
\]
and corresponding local coordinates 
\[
\bm{w}= (w_{-1}, w_{1}; \ldots; w_{-g},w_{g}),  
\]
of points $2g$ $(p_{-1}, p_1; \ldots;  p_{-g}, p_g)$ on the sphere  $\mathcal{S}_0$, 
consider the genus zero $2g$-point correlation function
\begin{align*}
Z^{(0)}(\bm{b,w})=&Z^{(0)}(b_{-1},w_{-1};b_{1},w_{1};\ldots;b_{-g},w_{-g};b_{g},w_{g})
\\
=&\prod_{a\in\Ip}\rho_{a}^{\wt(b_{a})}Z^{(0)}(\bbar_{1},w_{-1};b_{1},w_{1};\ldots;\bbar_{g},w_{-g};b_{g},w_{g}).
\end{align*}
where $\Ip=\{1,2,\ldots,g\}$. 
Let 
\[
\bm{b}_{+}=(b_{1},\ldots,b_{g}), 
\]
 denote an element of a $V$-tensor product $V^{\otimes g}$-basis with dual basis 
\[
 \bm{b}_{-}=(b_{-1}, \ldots,b_{-g}), 
\]
 with respect to the bilinear form 
$\langle \cdot, \cdot\rangle_{\rho_{a}}$ (cf. Appendix \ref{vertex}). 

 Let $w_{a}$ for $a\in\I$ be $2g$ 
 formal variables. One identify them  
 with the canonical Schottky parameters (see Appendix \ref{sec:Genusg}).   
 We can 
define the genus $g$ partition function as
\begin{align}\label{GenusgPartition}
Z_{V}^{(g)} =Z_{V}^{(g)}(\bm{w,\rho})
=\sum_{\bm{b}_{+}}Z^{(0)}(\bm{b,w}),
\end{align}
for 
\[
(\bm{w,\rho})=(w_{\pm 1},  \rho_{1}; \ldots; w_{\pm g},\rho_{g}). 
\]
This definition is motivated by the sewing relation \eqref{eq:SchottkySewing2}. 
\begin{remark} 
	Note that $Z_{V}^{(g)}$ depends on  $\rho_{a}$ via the dual vectors $\bm{b}_{-}$ as in \eqref{eq:bbar}. 
The genus $g$ partition function for the tensor product $V_{1}\otimes V_{2}$ of two vertex operator algebras $V_{1}$ and $V_{2}$ is  
\[
Z^{(g)}_{V_{1}\otimes V_{2}}=Z_{V_{1}}^{(g)} \; Z_{V_{2}}^{(g)}.
\]
\end{remark}
Now 
we recall 
a formal Zhu reduction expression for all
 genus $g$  Schottky $n$-point character functions.
One 
defines the genus $g$  formal $n$-point function for $n$ vectors 
$(v_{1},\ldots,v_{n}) \in V$ inserted $(y_{1},\ldots,y_{n})$ by
\begin{align}\label{GenusgnPoint}
Z_{V}^{(g)}(\bm{v,y}) =Z_{V}^{(g)}(\bm{v,y};\bm{w,\rho})
=
\sum_{\bm{b}_{+}}Z^{(0)}(\bm{v,y};\bm{b,w}),
\end{align}
where 
\[
Z^{(0)}(\bm{v,y};\bm{b,w})=Z^{(0)}(v_{1},y_{1};\ldots;v_{n},y_{n};b_{-1},w_{-1};\ldots;b_{g},w_{g}).
\] 
Let $U$ be a vertex operator subalgebra of $V$ where $V$ has a $U$-module decomposition
\[
V=\bigoplus_{\alpha\in A}
{W}_{\alpha},
\]
  for  $U$-modules $
{W}_{\alpha}$ and some indexing set $A$. 
Let 
\[
{W}_{\bm{\alpha}}=\bigotimes_{a=1}^{g} 
{W}_{\alpha_{a}}, 
\]
 denote a tensor product of $g$ modules 
\begin{align}\label{eq:Z_Walpha}
Z_{
{W}_{\bm{\alpha}}}^{(g)}(\bm{v,y}) =\sum _{\bm{b_{+}}\in 
{W}_{\bm{\alpha}}} Z^{(0)}(\bm{v,y};\bm{b,w}),
\end{align}
where here the sum is over a basis $\{\bm{b}_{+}\}$ for  $
{W}_{\bm{\alpha}}$. 
It follows that
\begin{align}
\label{eq:Z_WalphaSum}
Z_{V}^{(g)}(\bm{v,y})=\sum_{\bm{\alpha}\in\bm{A}}Z_{
{W}_{\bm{\alpha}}}^{(g)}(\bm{v,y}),
\end{align}
where the sum ranges over $\bm{\alpha}=(\alpha_{1},\ldots ,\alpha_{g}) \in \bm{A}$,  for $\bm{A}=A^{\otimes{g}}$. 
Finally, it is useful to define corresponding formal $n$-point correlation differential forms 
\begin{align}
\label{tupoy}
\F_{V}^{(g)}(\bm{v,y}) &=Z^{(g)}(\bm{v,y}) \; \bm{dy^{\wt(v)}},
\nn
\F_{
{W}_{\bm{\alpha}}}^{(g)}(\bm{v,y}) &=Z_{
{W}_{\bm{\alpha}}}^{(g)}(\bm{v,y})\; \bm{dy^{\wt(v)}},
\end{align}
where  
\[
\bm{dy^{\wt(v)}}=\prod_{k=1}^{n} dy_{k}^{\wt(v_{k})}.
\]
Recall notations and 
identifications given in Appendix \ref{derivation}. Then one has: 
\begin{proposition}
\label{theor:ZhuGenusg} 
The genus $g$ $(n+1)$-point formal differential 
$\F_{
{W}_{\bm{\alpha}}}^{(g)}(u,x;\bm{v,y})$ 
for a quasiprimary vector $u \in U$ of weight $\wt(u)=p$ inserted at a 
 point
$p_0$, with the coordinate $x$, and general vectors $(v_{1}, \ldots, v_{n})$ inserted at points 
 $p_1, \ldots, p_n$ with coordinates 
$(y_{1}, \ldots, y_{n})$ correspondingly, 
  respectively, satisfies the recursive identity
\begin{align}
\label{eq:ZhuGenusg}
&\F_{{
{W}_{\bm{\alpha}}}}^{(g)}(u, x; 
\bm{v,y}
)
\nn
& =  \sum_{k=1}^{n}\sum_{j\ge 0}\del^{(0,j)} \; \Psi_{p}(x,y_{k})\;
\F_{{
{W}_{\bm{\alpha}}}}^{(g)}(v_1, y_1; \ldots; u(j)v_{k}, y_{k}; \ldots; v_n, y_n)\; dy_{k}^{j}\; 
\nn
& \qquad + \sum_{a=1}^{g} \Theta_{a}(x) \; O^{
{W}_{\bm{\alpha}}}_{a}(u; 
\bm{v,y}
). 
\end{align}
Here $\del^{(0,j)}$ is given by 
\[
\del^{(i,j)}f(x,y)=\del_{x}^{(i)}\del_{y}^{(j)}f(x,y), 
\]
for a function $f(x,y)$, and $\del^{(0,j)}$ denotes partial derivatives with respect to $x$ and $y_j$. 
The forms 
$\Psi_{p}(x, y_{k} )\; dy_{k}^{j}$ given by 
 \eqref{psih}, 
 $\Theta_{a}(x)$ is of \eqref{thetanew}, and 
 $O^{
{W}_{\bm{\alpha}}}_{a}(u; 
\bm{v,y}
)$ of \eqref{oat}. 
\end{proposition}
\section{Schottky vertex operator cluster algebras} 
\label{cluster}
In this section we formulate Proposition \ref{pizda}  
 concerning a cluster vertex algebra associated to a vertex operator algebra in the case of 
 Schottky parameterization of a genus $g$ 
Riemann surface.  
That Proposition clarifies 
%
a cluster-like algebra structure for a vertex operator algebra.  
 Let us fix a strong-type (cf. Appendix \ref{vertex})  vertex operator algebra $V$.   
Chose $n+1$-marked points $p_0$ and $p_i$, $i=1, \ldots, n$ on a genus $g$  
Riemann surface formed by the Schottky parameterization 
(cf. Appendix \ref{sec:Genusg}). 
In the vicinity of each marked point $p_0$, $p_i$ define local coordinates $x$, $y_i$,  
with zero at points $p_0$, $p_i$ correspondingly.  

Consider $n$-tuples 
 of arbitrary states $v_i \in V$, and 
corresponding vertex operators 
\[
{\bf Y}( \bm{v, y} ) = \left( Y(v_1, y_1), \ldots, Y(v_n, y_n) \right), 
\] 
with coordinates $(y_1, \ldots, y_n)$, 
 around points $p_i$, $i=1, \ldots, n$.
\begin{definition}
\label{seeda}
We define a vertex operator cluster algebra seed    
\begin{equation}
\label{vertex_cluster}
\left( {\bm v}, {\mathbf Y}(\bm{v, y} 
), \F^{(g)}_n(\bm{v, y})  
\right), 
\end{equation}
where
\[
\F^{(g)}_n(\bm{v, y}) 
 = 
\F^{(g)}_{W_{\bm{\alpha}}} (\bm{v, y}),  
\] 
(and in particular $\F^{(g)}_{V} (\bm{v, y})$) is a genus $g$ 
  $n$-point character function $\F^{(g)}_n(\bm{v, y})$ 
 \eqref{tupoy}.
\end{definition}
The mutation is defined as follows: 
\begin{definition}
\label{muta}
For ${\bm v}$, we define the mutation ${\bm v}^\prime$ of 
 ${\bm v}$ in the direction $k \in 1, \ldots, n$ 
as 
\begin{equation}
\label{v-mutation}
{\bm v}^\prime= \mu_k(u, m) {\bm v}= 
\left(v_1, \ldots, F_k(u(m)). 
  {v_k}, \ldots,  v_n\right), 
\end{equation}
for some $m \ge 0$, and $V$-valued functions $F_k(v(m))$. 
Note that due to the property (\ref{lowertrun})  
we get a finite number of terms as a result of the action of $v(m)$ on $v_k$, $1 \le k \le n$. 
For the $n$-tuple of vertex operators we define
\begin{eqnarray}
\label{z-mutation}
{\mathbf Y} \left({\bm v}^\prime, {\bm y} \right) 
&=& \mu_k(u, m)  \;{\mathbf Y} \left( \bm{v, y}  
\right) 
\nn
&=& \left( Y(v_1, y_1), \ldots, 
Y( G_k(u(m)). v_k, y_k), \ldots, Y(v_n, y_n) \right),    
\end{eqnarray}
where $G_k(u(m))$ are other $V$-valued functions. 
For $u \in V$, $w\in \C$,   
 the mutation $\mu(u, x, \bm{y})$ of $\F^{(g)}_n (\bm{v, y})$,  
\begin{eqnarray}
 {\F'}_n^{(g)}(\bm{v, y}) 
 = \mu
(u, x, \bm{y})\; \F^{(g)}_n (\bm{v, y}),  
\label{matrix_rule}
\end{eqnarray}
is defined by summation  
over mutations in all possible directions $k$,  $1 \le k \le n$, with   
auxiliary functions $f(
m, k, x, \bm{y})$, 
$k \in 1, \ldots, n$: 
\begin{eqnarray}
\label{matrix_rule_n_plis_one}
 &&
 {\F'}^{(g)}_n( \bm{v, y}) 
\nn
&& \qquad 
= 
\sum\limits_{k=1}^{n} \sum\limits_{m \ge 0}  
 f(
  m, k, x, \bm{y}) \; 
\F^{(g)}_n(v_1, y_1; \ldots ; 
  H_k(u(m)). v_k, y_k;  \ldots; v_n, y_n), 
\nn
&&
\qquad \qquad \qquad \qquad +
 {\widetilde \F}^{(g)}_n(u, x; \bm{v, y}). 
\end{eqnarray}
 where ${\widetilde \F}^{(g)}_n(u, x; \bm{v, y})$,  
 denotes the higher terms in the genus $g$ 
 Zhu reduction formulas \eqref{eq:ZhuGenusg}, 
and $H_k(u(m))$ are $V$-valued functions.     
Then  \eqref{v-mutation}, \eqref{z-mutation}, \eqref{matrix_rule_n_plis_one} 
define  the mutation of the seed (\ref{vertex_cluster}). 
%
\end{definition}
\begin{definition}
Definitions \ref{seeda}--\ref{muta},  the genus $g$ Zhu reduction procedure,  and involutivity 
condition for mutation  
determine the structure of a genus $g$ vertex operator cluster algebra ${\mathcal C}{\mathcal G}^{(g)}_n$ of dimension $n$.  
%
We call the full vertex operator cluster algebra the union $\bigcup_{n \ge 0} \; {\mathcal C}{\mathcal G}^{(g)}_n$. 
\end{definition} 
\begin{remark}
Exchange matrix of ordinary cluster algebras is replaced in this construction with genus $g$ Schottky 
characters for a vertex operator algebra. These are higher genus generalizations of matrix 
elements at genus zero \cite{FHL}, and traces of vertex operator algebra modules at genus one
\cite{Z}. 
\end{remark}
Using \eqref{eq:ZhuGenusg}, we obtain in 
\eqref{v-mutation}, \eqref{z-mutation}, and \eqref{matrix_rule_n_plis_one}:  
\begin{eqnarray}
f( 
 m, k, x, \bm{y}) &=& \del^{(0,j)} \Psi_{p}(x, y_{k})\; dy_{k}^{j},  
\nn
{\widetilde \F}^{(g)}_n(u, x; \bm{v, y}
) &=&
 \sum_{a=1}^{g} \Theta_{a}(x)\; O^{
{W}_{\bm{\alpha}}}_{a}(u; \bm{v, y} 
%
%
). 
\end{eqnarray}
Next we provide an 
example of the Schottky vertex operator cluster algebra.    
We formulate
\begin{proposition}
\label{pizda}
For a vertex operator algebra $V$ such that $\dim V_k=1$, $k\in \Z$, 
with $u=\one_V$,  $w\in \C$, and 
\[
F_k(u(m)).v=G_k(u(m)).v= \xi_{u, v} u(-1).v, 
\]
\[
  H_k(u(m))= u(m), 
\]
for $ m \ge 0$, and $\xi_{u, v}\in \C$, $\xi_{u, v}^2=1$, depending on $u$ and $v$, 
in \eqref{v-mutation}, \eqref{z-mutation}, and \eqref{matrix_rule_n_plis_one},  
the mutation 
\[
\mu 
= 
\left( \mu_k(\one_V, -1), \mu_k(\one_V, -1),  \mu(\one_V, x, \bm{y}) \right),  
\]
\begin{equation}
\label{mutation}
\left( {\bm v}^\prime, {\mathbf Y} ({\bm v}^\prime, {\bm y}), 
{\F'}^{(g)}_n({\bm v}^\prime, {\bm y}) 
\right)   
=\mu
\;\left( {\bm v}, {\mathbf Y}(\bm{v, y}) 
, \F^{(g)}_n(\bm{v, y}) 
\right), 
 \end{equation}
defined by (\ref{v-mutation}), (\ref{z-mutation}), (\ref{matrix_rule_n_plis_one}) 
is an involution,   
i.e., 
\[
\mu
\; \mu 
= {\rm Id}. \;  
\]
\end{proposition}
\begin{proof}
According to \eqref{modeaction1} 
for $u= \one_V \in V_0$,  and $v_k \in V_l$, $1 \le k \le n$, $l \in \Z$, 
\[
u(-1) u(-1).v_k: V_l \rightarrow V_{l}. 
\]
Due to the genus $g$ 
 Zhu reduction formula  and \eqref{eq:ZhuGenusg}, we have  
\begin{eqnarray}
 && {\F'}_n^{(g)}(\bm{v, y}) 
=  
\mu
({\mathbf 1}_V, x, \bm{y}) \; \F^{(g)}_{n} ( \bm{v, y})
\nonumber
\\
\nonumber
&& \; =
 \sum\limits_{k=1}^{n} \sum\limits_{m \ge 0}  
 f(
m, k, x, \bm{y}) \; 
\F^{(g)}_n(v_1, z_1; \ldots ; {\mathbf 1}_V[m]. v_k, z_k;  \ldots; v_n, z_n; \tau_1, \tau_2, \epsilon)
\nn
&&
\qquad \qquad \qquad + {\widetilde \F}^{(g)}_n({\mathbf 1}_V, x; \bm{v, y}) 
\nn
&& 
\; 
= \F^{(g)}_{n+1}({\bf 1}_V, x; \bm{v, y})  
=  \F^{(g)}_n(\bm{v, y}). 
\label{matrix_rule_n_plis_one_1}
\end{eqnarray}
Thus, in this case, the mutation $\mu
$ is 
an involution. 
\end{proof}
%

\section*{Acknowledgement} 
The author would like to thank 
H. V. L\^e, and A. Lytchak  
 for related discussions. 
Research of the author was supported by the GACR project 18-00496S and RVO: 67985840. 

\section{Appendix: Vertex Operator Algebras} 
\label{vertex}
A vertex operator algebra \cite{B, DL, FHL, FLM, K, LL, MN} is determined by a quadruple  
$(V,Y,\mathbf{1}_V,\omega)$,  
where 
is a 
linear space 
endowed with a $\mathbb Z$-grading with 
\[
V=\bigoplus_{r \in {\mathbb Z}} V_{r},   
\]
with $\dim V_{r}<\infty$.
  The state  ${\mathbf 1}_V \in V_{0}$, $\mathbf{1}_V \ne  0$,  is the vacuum vector and 
$\omega\in V_{2}$ is the conformal vector with properties described below.
The vertex operator $Y$ is a linear map 
\[
Y: V\rightarrow \mathrm{End}(V)\left[\left[z,z^{-1}\right]\right], 
\]
for formal variable $z$ so that for any vector $u\in V$ we have a vertex operator  
\begin{equation}
Y(u,z)=\sum_{n\in {\mathbb Z}}u(n)z^{-n-1}.  
\label{Ydefn}
\end{equation}
The linear operators (modes) $u(n):V\rightarrow V$ satisfy creativity 
\begin{equation}
Y(u,z){\mathbf 1}_V= u +O(z),
\label{create}
\end{equation}
and lower truncation 
\begin{equation}
u(n)v=0,
\label{lowertrun}
\end{equation}
conditions for each $u$, $v\in V$ and $ n\gg 0$. 
For the conformal vector $\omega$ one has 
\begin{equation}
Y(\omega ,z)=\sum_{n\in {\mathbb Z}}L(n)z^{-n-2},  \label{Yomega}
\end{equation}
where $L(n)$ satisfies the Virasoro algebra for some central charge $C$ 
\begin{equation}
[\,  L(m),L(n)\, ]=(m-n)L(m+n)+\frac{C}{12}(m^{3}-m)\delta_{m,-n}{\rm Id}_V, 
\label{Virasoro}
\end{equation}
where ${\rm Id}_V$ is identity operator on $V$. 
Each vertex operator satisfies the translation property 
\begin{equation}
\partial_z Y(u,z)= Y\left(L(-1)u,z\right).  
\label{YL(-1)}
\end{equation}
The Virasoro operator $L(0)$ provides the ${\mathbb Z}$-grading with 
\[
L(0)u=ru, 
\]
 for 
$u\in V_{r}$, $r\in {\mathbb Z}$. 
Finally, the vertex operators satisfy the Jacobi identity
\begin{eqnarray}
\nonumber
 &&z_0^{-1}\delta\left( \frac{z_1 - z_2}{z_0}\right) Y (u, z_1 )Y(v , z_2)    
  - 
z_0^{-1} \delta \left( \frac{z_2 - z_1}{-z_0}\right) Y(v, z_2) Y(u , z_1 ) 
\\&&
\label{ vertex operator algebraJac}
 \qquad \qquad \qquad
= z_2^{-1}    
\delta\left( \frac{z_1 - z_0}{z_2}\right)
Y \left( Y(u, z_0)v, z_2\right).  
\end{eqnarray} 
These axioms imply 
locality, skew-symmetry, associativity and commutativity conditions:
\begin{eqnarray}
&(z_{1}-z_{2})^N
Y(u,z_{1})Y(v,z_{2}) 
= 
(z_{1}-z_{2})^N
Y(v,z_{2})Y(u,z_{1}),&
\label{Local}
\\
\nonumber
\\
&Y(u,z)v = 
e^{zL(-1)}Y(v,-z)u,&
\label{skew}
\nonumber
\\
\nonumber
\\
\nonumber
&(z_{0}+z_{2})^N Y(u,z_{0}+z_{2})Y(v,z_{2})w = (z_{0}+z_{2})^N Y(Y(u,z_{0})v,z_{2})w,&
\label{Assoc}
\\
\nonumber
\\
&u(k)Y(v,z)-
 Y(v,z)u(k)
= \sum\limits_{j\ge 0}  \left( k \atop j \right)
Y(u(j)v,z)z^{k-j},&
\label{Comm}
\end{eqnarray}
for $u$, $v$, $w\in V$ and integers $N\gg 0$. 
For $v={\mathbf 1}_V$  
one has
\begin{equation}
\label{edinica}
Y({\mathbf 1}_V, z)={\rm Id}_V. 
\end{equation}
Note also that modes of homogeneous states are graded operators on $V$, i.e., for $v \in V_k$,  
\begin{eqnarray}
\label{modeaction1}
 v(n): V_m \rightarrow V_{m+k-n-1}.
\end{eqnarray}
 In particular, let us define the zero mode $o(v)$ of a state of weight $wt(v)=k$, i.e., $v \in V_k$,
as 
\begin{equation}
\label{zero_mode}
o(v) = v(wt (v) - 1), 
\end{equation}
 extending to $V$ additively. 
\begin{definition}
Given a vertex operator algebra $V$, one defines the 
{adjoint vertex operator} with respect to $\alpha \in \C$,  by
\begin{eqnarray}
Y^\dagger (u,z) &=&\sum_{n\in\Z}u^\dagger(n) z^{-n-1} 
\nn
&=& 
Y\left(\exp\left( \frac{z}{\alpha} 
L(1) \right)\left(- \frac{\alpha}
{z^2}\right)^{L(0)}u, \frac{\alpha}{z}\right),
\label{eq:Ydag}
\end{eqnarray} 
associated with the formal M\"obius map \cite{FHL}
\[
z \mapsto \frac{\alpha}{z}.  
\]
\end{definition}
\begin{definition}
An element $u\in V$ is called quasiprimary if 
\[
L(1)u=0. 
\]
\end{definition}
For quasiprimary $u$  
of weight $\wt(u)$ one has
\begin{align*}
u^\dagger(n) = (-1)^{\wt(u)} \alpha^{n+1-\wt(u)} u(2\wt(u)-n-2).
\end{align*}
\begin{definition}
A bilinear form 
\[
\langle . , . \rangle : V \times V\rightarrow \C, 
\]
 is called  {invariant} if \cite{FHL, Li}
\begin{align}
\langle Y(u,z)a,b \rangle = \langle a,Y^\dagger(u,z)b \rangle, 
\label{eq:bilform}
\end{align}
for all 
$a$, $b$, $u\in V$.
\end{definition}
Notice that the adjoint vertex operator $Y^\dagger(.,.)$ as well as the bilinear form $\langle. ,. \rangle$,  
 depend on $\alpha$.
 In terms of modes, we have 
\begin{align}
\langle u(n)a,b \rangle = \langle a,u^\dagger(n)b \rangle.
\label{eq:undag}
\end{align} 
Choosing $u=\omega$, and for $n=1$ implies 
\[
\langle L(0)a,b \rangle = \langle a,L(0)b \rangle.
\]
Thus, 
\[
\langle a, b \rangle=0, 
\]
 when $\wt(a) \neq \wt(b)$.  
\begin{definition}
A vertex operator algebra is called of strong-type 
if 
\[
V_{0} = \C\mathbf{1}_V, 
\]
and $V$ is simple and self-dual, i.e., $V$ is isomorphic to the dual module $V^\prime$ as a $V$-module.  
\end{definition}
It is proven in \cite{Li} 
that a strong-type vertex operator algebra $V$ has a unique invariant non-degenerate  
bilinear form up to normalization. 
This motivates 
\begin{definition}
\label{def:LiZ}
The form 
$\langle  ., .\rangle$
 on a strong-type vertex operator algebra $V$ is the unique invariant bilinear form $\langle . , . \rangle$  
normalized by 
\[
\langle \mathbf{1}_V, \mathbf{1}_V \rangle = 1.
\]
 \end{definition}
Given a vertex operator algebra $(V,Y(.,.),\mathbf{1},\omega)$, one can find an isomorphic vertex operator algebra 
 $(V,Y[.,.],\mathbf{1},\omt)$ 
called \cite{Z} the 
{square-bracket} vertex operator algebra. 
 Both algebras have the same underlying vector space $V$, vacuum vector $\mathbf{1}_V$, and central charge. 
The vertex operator $Y[., .]$ is determined by 
\begin{align*}
Y[v,z] = \sum_{n\in \Z} v[n] z^{-n-1} = Y\left(q_z^{L(0)}v, q_z-1\right).
\end{align*} 
The new square-bracket conformal vector is 
\[
\omt = \omega - \frac{c}{24}\mathbf{1},
\]
with the vertex operator
\[
Y[\omt,z] = \sum_{n\in \Z} L[n] z^{-n-2}. 
\]
The square-bracket Virasoro  operator mode 
$L[0]$ provides an alternative $\Z$--grading on $V$, i.e.,   
$\wt[v]=k$ if 
\[
L[0]v = kv, 
\]
 where $\wt[v]=\wt(v)$ for 
 primary $v$, and $L(n)v=0$ for all $ n>0$. 
 We can similarly define a square-bracket
bilinear form 
$\langle . , . \rangle_{\mathrm{sq}}$.

Next we recall a lemma from \cite{TW}. 
The bilinear form 
$\langle \cdot, \cdot\rangle$ of 
is invertible and that
\[
\langle u, v \rangle=0,
\]
Let $\{ b\}$ be a homogeneous basis for  $V$ with 
the dual basis $\{\bbar \}$. 
\begin{lemma}
\label{AdjointLemma}
	For $u$ quasiprimary of weight $p$ we have
	\begin{align}
	\label{eq:adjointrel}
	\sum_{b\in V_n}
 \left( u(m)b \right)\otimes \bbar =\sum_{b\in V_{n+p-m-1}} b\otimes \left(u^\dagger(m)\bbar \right).
	\end{align}
\end{lemma}
\begin{remark}
	\label{rem:adjoint}
Suppose that $U$ is a vertex operator subalgebra of $V$ and  $
{W}\subset V$ is a $U$-module. For $u\in U$ and homogeneous $
{W}$-basis $\{ w\}$  we may then extend  \eqref{eq:adjointrel} to obtain  
	\begin{align*}
\sum_{w\in 
{W}_{n}} \left(u(m)w\right)\otimes \overline{w}=\sum_{w\in 
{W}_{n+p-m-1}} w\otimes \left(u^\dagger(m)\overline{w}\right).
\end{align*} 
\end{remark}
For the Schottky setup we have the following properties associated to the $\rho$-sewing. 
For each $a\in\Ip$, let $\{b_{a}\}$  denote a homogeneous  $V$-basis and let $\{\bbar _{a}\}$ be the  dual basis with 
$\langle \cdot, \cdot\rangle_{1}$,  i.e.,  with $\rho=1$.  
Define
\begin{align}
\label{eq:bbar}
b_{-a}=\rho_{a}^{\wt(b_{a})}\bbar _{a},\quad a\in\Ip,
\end{align}
for a formal $\rho_{a}$. We then 
identify $\rho_a$ with a Schottky sewing parameter. 
Then $\{b_{-a}\}$ is a dual basis for the bilinear form 
$\langle \cdot,\cdot\rangle_{\rho_{a}}$  with adjoint modes 
\begin{align}\label{RhoAdjoint}
u^{\dagger}_{\rho_{a}}(m)=(-1)^{p}\rho_{a}^{m-p+1}u(2p-2-m),
\end{align}
for $u$ quasiprimary of weight $p$.

\section{Appendix: Definition of a cluster algebra} 
\label{definition}
Let us first recall the notion of a cluster algebra \cite{FZ1,  FZ2, FZ3}   
following of\cite{Sch}.     
We consider commutative cluster algebras of rank $n$. 
The set of all cluster variables is constructed recursively from	
an initial set of $n$ cluster variables using mutations.
Every mutation defines a new cluster variable as a rational function of the cluster variables 
constructed previously. 
Thus, recursively, every cluster variable is a certain rational function in
the initial $n$ cluster variables.
These rational functions are Laurent polynomials \cite{FZ1}.
 A cluster algebra is determined by its initial 
seed which consists of a cluster 
\[
{\bf x} = 
(x_1, \ldots, x_n),  
\] 
of algebraically independent set of generators,  
 a coefficient tuple
\[
{\bf y} =
(y_1, \ldots, y_n), 
\]
 and a skew-symmetrizable $n \times n$ 
integer exchange matrix 
\[
B = \left(b_{ij}\right),
\]
 i.e., $b_{i,j} = -b_{j,i}$.
The coefficients $\left\{y_1, \ldots, y_n \right\}$ are taken in a torsion free 
abelian group ${\mathbb P}$. 
The mutation in direction $k$ defines a new cluster 
\begin{equation}
 x'_k x_k = y^+\prod\limits_{b_{k,i}>0} x_i^{b_{k,i}} + y^-\prod\limits_{b_{k,i}<0} x_i^{-b_{k,i}},  
\end{equation}
where $y^\pm$ are certain monomials in $(y_1, \ldots , y_n)$.
 Mutations also transform the coefficient tuple $y$ and the matrix $B$. 

If $\zeta$ is any cluster variable, then $u$ is obtained from the initial cluster 
$
\bm{x}
$
 by a sequence
of mutations, then \cite{FZ1} $\zeta$ can be written as a Laurent polynomial in variables
$(x_1, \ldots, x_n)$, that is,
\begin{equation}
\label{ration}
f(\bm{x} 
)= 
\zeta \; {\prod\limits_{i=1}^{n} x_i^{d_i}}, 
\end{equation}
for some $d_i$, 
where $f(\bm{x}
)$ is a polynomial with coefficients in the group 
ring ${{\mathbb Z}{\mathbb P}}$ 
of the coefficient group ${\mathbb P}$.
A cluster algebra is of finite type if it has only a finite number of seeds. 
In\cite{FZ2}
it was shown 
that cluster algebras of finite type can be classified in terms of the Dynkin diagrams of 
finite-dimensional simple Lie algebras.
\subsection{Formal definition} 
\label{cluster_algebra_defintion}
Let ${\mathbb P}$ be an abelian group with binary operation $\oplus$, 
 ${\mathbb Z}{\mathbb P}$ be the group ring of ${\mathbb P}$,  
and let ${\mathbb Q}{\mathbb P}(\bm{x} 
)$ be the field of rational functions
in $n$ variables with coefficients in ${\mathbb Q}{\mathbb P}$.
\begin{definition}
A seed is a triple $({\mathbf x}, {\mathbf y}, B)$, where
 ${\mathbf x}  = \left\{x_1, \ldots, x_n\right\}$ is a basis 
of ${\mathbb Q}{\mathbb P}\left(x_1, \ldots, x_n\right)$, 
 ${\mathbf y}  = \left\{y_1, \ldots, y_n\right\}$, 
is an $n$-tuple of elements $y_i \in {\mathbb P}$, and
$B$ is a skew-symmetrizable matrix. 
\end{definition}
\begin{definition}
Given a seed 
\[
\left({\mathbf x}, {\mathbf  y}, B\right), 
\]
 its mutation 
$\mu_k ({\mathbf x},{\mathbf  y}, B)$ in direction $k$ 
is a new seed
 $({\mathbf x}^\prime, {\mathbf  y}^\prime, B^\prime)$
defined as follows. Let $[x]_+ = max(x, 0)$. 
Then we have 
$B^\prime = (b^\prime_{ij} )$ with 
\begin{eqnarray}
b^\prime_{ij}= 
\left[
\begin{array}{l}
b_{ij} \; {\rm for} \; i = k \; {\rm or} \; j = k,
\\
b_{ij} + [-b_{ik}]_+ b_{kj} + b_{ik}[b_{kj} ]_+, \;\;  {\rm otherwise}.  
\end{array}
\right.
\end{eqnarray}
For new coefficients 
${\mathbf y}^\prime = \left( y^\prime_1, \ldots, y^\prime_n \right)$, 
with
 \begin{eqnarray}
y^\prime_j=  
\left[
\begin{array}{l}
y^{-1}_k \;  {\rm if} \;  j = k,
\\
y_j y_k^{[b_{kj}]_+} (y_k \oplus 1)^{-b_{kj}}\;  {\rm if}  \; j \ne k, 
\end{array}
\right.
\end{eqnarray} 
and 
${\bf x} = 
(x_1, \ldots, x_n)$, 
 where
\begin{equation}
(y_k \oplus 1) \; x_k \; x^\prime_k = 
y_k \prod\limits_{i=1}^n x_i^{[b_{ik}]_+} 
 +
\prod\limits_{i=1}^n
x_i^{[-b_{ik}]_+}. 
\end{equation} 
\end{definition}
Mutations are involutions, i.e., 
\[
\mu_k \; \mu_k \; ({\bf x}, {\bf  y}, B)=({\bf x}, {\bf  y}, B).
\]
\section{Appendix: The Schottky uniformization of Riemann surfaces}
\label{sec:Genusg} 
In this appendix we recall the  Schottky uniformization of Riemann surfaces \cite{TW}. 
Consider a compact marked  Riemann surface $\Sg$ of genus $g$, e.g., \cite{FK, Mu, Fa, Bo},  
with canonical homology basis 
$\alpha_{a}$, $\beta_{a}$ for $a\in\Ip=\{ 1,2, \ldots,g \}$.
We recall  
the  construction of a genus $g$ Riemann surface $\Sg$ using the Schottky uniformization where we sew $g$ handles to the Riemann sphere 
\[
\Szero\cong\Chat=\C\bigcup \{\infty\}, 
\]
 e.g., \cite{Fo, Bo}. 
Every Riemann surface can be non-uniquely Schottky uniformized \cite{Be2}.
For  $a\in\I=\{\pm 1,\pm 2,\ldots,\pm g\}$, 
  let $\calC_{a}\subset \Szero$ be $2g$ non-intersecting Jordan curves. 
For 
$z \in \calC_{a}$, $z'\in \calC_{-a}$, $W_{\pm a}\in\Chat$,  
$a\in\Ip$, and 
$q_{a}$ with 
\[
0<|q_{a}|<1, 
\]
let curves be identified by the sewing relation 
\begin{align}\label{eq:SchottkySewing}
\frac{z'-W_{-a}}{z'-W_{a}}\cdot\frac{z-W_{a}}{z-W_{-a}}=q_{a}. 
%
\end{align}
For $a\in\Ip$,   Introduce
\begin{align}
\sigma_{a}=(W_{-a}-W_{a})^{-1/2}\begin{pmatrix}
1 & -W_{-a}\\
1 & -W_{a}
\end{pmatrix}, 
\end{align}
and 
\begin{align}
\label{eq:gammaa}
\gamma_{a}=\sigma_{a}^{-1}
\begin{pmatrix}
q_{a}^{1/2} &0\\
0 &q_{a}^{-1/2}
\end{pmatrix}
\sigma_{a},
\end{align}
Thus  
\[
z'=\gamma_{a}z. 
\]
Note that 
\[
\sigma_{a}(W_{-a})=0, 
\]
 and 
\[
\sigma_{a}(W_{a})=\infty, 
\]
 are, respectively, attractive and repelling fixed points of the map
\[
Z \rightarrow Z'=q_{a} Z, 
\]
 for  
\[
Z=\sigma_{a} z, 
\]
 and 
\[
Z'=\sigma_{a} z'.
\]
Here $W_{-a}$ and  $W_{a}$  are the corresponding fixed points for $\gamma_{a}$. 
One identifies the standard homology cycles 
$\alpha_{a}$  with $\calC_{-a}$ and  $\beta_{a}$ with a path connecting  $z\in  \calC_{a}$ to 
\[
z'=\gamma_{a}z, \in  \calC_{-a}.
\]
and $z'\in  \calC_{-a}$.
\begin{definition}
The genus $g$ Schottky group $\Gamma$  is   the free group with generators $\gamma_{a}$.
Define 
\[
\gamma_{-a}=\gamma_{a}^{-1}.
\]
 The independent elements of $\Gamma$ are reduced words of length $k$
 of the form
\[
\gamma=\gamma_{a_{1}}\ldots \gamma_{a_{k}}, 
\]
 where $a_{i}\neq -a_{i+1}$ for each $i=1,\ldots ,k-1$. 
\end{definition}
 Let $\Lambda(\Gamma)$ denote the  limit set  of $\Gamma$, i.e.,
 the set of limit points of the action of $\Gamma$ on $\Chat$. 
Then 
\[
\Sg \simeq \Omo/\Gamma 
\]
where 
\[
\Omo=\Chat-\Lambda(\Gamma).
\]
  We let $\D\subset\Chat$  denote the standard connected fundamental 
region with oriented boundary curves $\calC_{a}$.  
Define  
\[
w_{a}=\gamma_{-a}.\infty.
\]
 Using \eqref{eq:SchottkySewing} we find 
\begin{align}
\label{eq:wa}
w_{a}=\frac{W_{a}-q_{a}W_{-a}}{1-q_{a}}, 
\end{align}
for $a\in\I$. 
where we define $q_{-a}=q_{a}$.
Then \eqref{eq:SchottkySewing} is equivalent to
\begin{align}\label{eq:SchottkySewing2}
(z'-w_{-a})(z-w_{a})=\rho_{a},
\end{align}
with 
\begin{align}
\label{eq:rhoa}
\rho_{\pm a}=-\frac{q_{a}(W_{a}-W_{-a})^{2}}{(1-q_{a})^{2}}.
\end{align}
\eqref{eq:SchottkySewing2} implies
\begin{align}
\label{eq:gamma_a.z}
\gamma_{a}z=w_{-a}+\frac{\rho_{a}}{z-w_{a}}.
\end{align}
 Let $\Delta_{a}$ be the disc with centre $w_{a}$ and radius $|\rho_{a}|^{\frac{1}{2}}$. 
One 
chooses the Jordan curve $\calC_{a}$ to be the  boundary of $\Delta_{a}$. Then 
$\gamma_{a}$ maps the exterior (interior) of $\Delta_{a}$  to the interior (exterior) of $\Delta_{-a}$ since
 \begin{align*}
 |\gamma_{a}z-w_{-a}||z-w_{a}|=|\rho_{a}|.
 \end{align*}
The discs $\Delta_{a}$, $\Delta_{b}$ are non-intersecting if and only if
 \begin{align}
 \label{eq:JordanIneq}
 |w_{a}-w_{b}|>|\rho_{a}|^{\frac{1}{2}}+|\rho_{b}|^{\frac{1}{2}},
%
 \end{align} 
for all $a\neq b$. 
One defines $\mathfrak{C}_{g}$ to be the set 
\[
 \{ (w_{a},w_{-a},\rho_{a})|a\in\Ip\}\subset \C^{3g}, 
\]
 satisfying \eqref{eq:JordanIneq}. We refer to $\mathfrak{C}_{g}$ as the Schottky parameter space.
  
The relation \eqref{eq:SchottkySewing} is M\"obius  invariant for 
\[
\gamma =\left(
\begin{array}{cc}
A&B\\C&D
\end{array}
\right)\in\SL_{2}(\C), 
\]
 with 
\[
(z,z',W_{a},q_{a})\rightarrow(\gamma z,\gamma z',\gamma W_{a},q_{a}), 
\]
 giving
an $\SL_{2}(\C)$ action on $\mathfrak{C}_{g}$ as follows
\begin{align}
\label{eq:Mobwrhoa}
&\gamma:(w_{a},\rho_{a})\mapsto 
\nn
& \qquad  \left(	\frac { \left( Aw_{a}+B \right)  \left( Cw_{-a}+D \right) -\rho_{a}
	\,AC}{ \left( Cw_{a}+D \right)  \left( Cw_{-a}+D \right) -\rho_{a}\,{
		C}^{2}},
{\frac {\rho_{a}}{ \left(  \left( Cw_{a}+D \right)  \left( Cw_{-a}+D
		\right) -\rho_{a}\,{C}^{2} \right) ^{2}}}\right).
\end{align}
\begin{definition}
One defines the  Schottky space as
\[
\mathfrak{S}_{g}=\mathfrak{C}_{g}/\SL_{2}(\C), 
\]
\end{definition}
This 
 provides a natural covering space for the moduli space of genus $g$ 
Riemann surfaces (of dimension 1 for $g=1$ and $3g-3$ for $g\ge 2$).  

\section{Appendix: coefficient functions in the Zhu reduction formula}
\label{derivation}
For purposes of the formula \eqref{eq:ZhuGenusg} 
%
 we recall here certain definitions \cite{TW}. 
Define a column vector 
\[
X=(X_{a}(m)), 
\]
 indexed by $ m\ge 0$ and $ a\in\I$  with components 
\begin{align}\label{XamDef}
X_{a}(m)=\rho_{a}^{-\frac{m}{2}}\sum_{\bm{b}_{+}}Z^{(0)}(\ldots;u(m)b_{a}, w_{a};\ldots),
\end{align}
and a row vector 
\[
p(x)=(p_{a}(x,m)), 
\]
  for $m\ge 0, a\in\I$  with components 
\begin{align}\label{eq:pdef}
p_{a}(x,m)=\rho_{a}^{\frac{m}{2}}\del^{(0,m)}\psi_{p}^{(0)}(x,w_{a}).
\end{align}
Introduce the  
column vector 
\[
G=(G_{a}(m)), 
\]
 for $m\ge 0, a\in\I$, given by 
\begin{align*}
G=\sum_{k=1}^{n}\sum_{j\ge 0}\del_{k}^{(j)} \; q(y_{k})\; Z_{V}^{(g)}(v_1, y_1; \ldots; u(j)v_{k},y_{k};
\ldots; v_n, y_n),
\end{align*}
where $q(y)=(q_{a}(y;m))$,  for $m\ge 0$, $a\in\I$, is a column vector with components 
\begin{align}
\label{eq:qdef}
q_{a}(y;m)=(-1)^{p}\rho_{a}^{\frac{m+1}{2}}\del^{(m,0)}\psi_{p}^{(0)}(w_{-a},y),
\end{align}
and 
\[
R=(R_{ab}(m,n)), 
\]
 for $m$, $n\ge 0$ and $a$, $b\in\I$ is a doubly indexed matrix with components 
\begin{align}
R_{ab}(m,n)=\begin{cases}(-1)^{p}\rho_{a}^{\frac{m+1}{2}}\rho_{b}^{\frac{n}{2}}\del^{(m,n)}\psi_{p}^{(0)}(w_{-a},w_{b}),&a\neq-b,\\ 
(-1)^{p}\rho_{a}^{\frac{m+n+1}{2}}\E_{m}^{n}(w_{-a}),&a=-b, 
\end{cases}
\label{eq:Rdef}
\end{align}
where
\begin{align}
\label{eq:Ejt}
\E_{m}^{n}(y)=\sum_{\ell=0}^{2p-2}\del^{(m)}f_{\ell}(y)\;\del^{(n)}y^{\ell}, 
\end{align}
\begin{align}\label{PsiDef}
\psi_{p}^{(0)}(x,y)=\frac{1}{x-y}+\sum_{\ell=0}^{2p-2}f_{\ell}(x)y^{\ell},
\end{align}
for {any} Laurent series $f_{\ell}(x)$ for $\ell=0,\ldots ,2p-2$.
Define the doubly indexed matrix $\Delta=(\Delta_{ab}(m,n))$ by
\begin{align}
\Delta_{ab}(m,n)=\delta_{m,n+2p-1}\delta_{ab}. 
\label {eq:Deltadef}
\end{align}
Denote by 
\[
\widetilde{R}=R\Delta,
\]
and the formal inverse $(I-\widetilde{R})^{-1}$ is given by 
\begin{align}
\label{eq:ImRinverse}
\left(I-\widetilde{R}\right)^{-1}=\sum_{k\ge 0}\widetilde{R}^{\,k}.
\end{align}
Define 
$\chi(x)=(\chi_{a}(x;\ell))$ and 
\[
o(u;\bm{v,y})=(o_{a}(u;\bm{v,y};\ell)), 
\]
 are 
{finite} row and column vectors indexed by 
$a\in\I$, $0\le \ell\le 2p-2$ with
\begin{align}
\label{eq:chiadef}
\chi_{a}(x;\ell)&=\rho_{a}^{-\frac{\ell}{2}}(p(x)+\widetilde{p}(x)(I-\widetilde{R})^{-1}R)_{a}(\ell),
\\
\label{LittleODef}
o_{a}(\ell)&=o_{a}(u;\bm{v,y};\ell)=\rho_{a}^{\frac{\ell}{2}}X_{a}(\ell),
\end{align}
and where 
\[
\widetilde{p}(x)=p(x)\Delta. 
\]
  $\psi_{p}(x,y)$ is defined by
\begin{align}
\label{eq:psilittleN}
\psi_{p}(x,y)=\psi_{p}^{(0)}(x,y)+\widetilde{p}(x)(I-\widetilde{R})^{-1}q(y).
\end{align}
For each $a \in\Ip$ we define a vector
\[
\theta_{a}(x)=(\theta_{a}(x;\ell) ), 
\]
 indexed by 
$0\le \ell\le  2p-2$ with components
\begin{align}\label{eq:thetadef}
\theta_{a}(x;\ell) = \chi_{a}(x;\ell)+(-1)^{p }\rho_{a}^{p-1-\ell}\chi_{-a}(x;2p-2-\ell).
\end{align}
Now define the following vectors of formal differential forms
\begin{align}
\label{eq:ThetaPQdef}
 P(x) =p(x) \; dx^{p},
\nn
 Q(y)=q(y)\; dy^{1-p},
\end{align}
with 
\[
\widetilde{P}(x)=P(x)\Delta.
\]
 Then with 
\begin{equation}
\label{psih}
\Psi_{p} (x,y) =\psi_{p}(x,y) \;dx^{p}\; dy^{1-p}, 
\end{equation}
 we have
\begin{align}\label{GenusgPsiDef}
\Psi_{p}(x,y)=\Psi_{p}^{(0)}(x,y)+\widetilde{P}(x)(I-\widetilde{R})^{-1}Q(y).
\end{align}
Defining   
\begin{equation}
\label{thetanew}
\Theta_{a}(x;\ell) =\theta_{a}(x;\ell)\; dx^{p}, 
\end{equation}
  and 
\begin{equation}
\label{oat}
O_{a}(u; \bm{v,y};\ell) = o_{a}(u; \bm{v,y};\ell) \; \bm{dy^{\wt(v)}}, 
\end{equation}
%
\begin{remark}
	\label{rem:MainTheorem}
The $\Theta_{a}(x)$, and $\Psi_{p}(x,y )$ coefficients depend on $p=\wt(u)$ but  are otherwise independent of the vertex operator algebra $V$. 
Note that for a 1-point function, \eqref{eq:ZhuGenusg} implies 
\begin{align}
\label{eq:1ptfun}
\F_{V}^{(g)}(u,x)&=\sum_{a=1}^{g} \Theta_{a}(x)\; O_{a}(u).
\end{align}
\end{remark}



\begin{thebibliography}{0}


\bibitem
{A} Ahlfors, L. 	
Some remarks on  Teichm\"uller's space of Riemann surfaces,
Ann.Math. \textbf{74} (1961) 171-191.

\bibitem
{AGMV} Alvarez-Gaume, L., Moore, G. and Vafa, C.
Theta functions, modular invariance, and strings,
Comm.Math.Phys. \textbf{106}  1-?40 (1986).


\bibitem
{Ba} 
Baker, H.F.
\textit{Abel's Theorem and the Allied Theory Including the Theory of Theta Functions}, 
Cambridge University Press (Cambridge, 1995).


\bibitem{BPZ} Belavin,~A.,  Polyakov,~A. and Zamolodchikov,~A.:
 Infinite conformal symmetry in two-dimensional quantum field theory.
 Nucl.~Phys. \textbf{B241} 333--380 (1984). 




\bibitem
{Be1} 
Bers, L.
Inequalities for finitely generated Kleinian groups, 
J.Anal.Math. \textbf{18} 23--41 (1967).

\bibitem
{Be2} 
Bers, L.
Automorphic forms for Schottky groups, 
Adv.Math. \textbf{16}  332?-361 (1975).



\bibitem
{Bo} 
Bobenko, A.
Introduction to compact Riemann surfaces,
in  \textit{Computational Approach to Riemann Surfaces}, 
edited Bobenko, A. and Klein, C.,  Springer-Verlag (Berlin, Heidelberg, 2011).



\bibitem{B} Borcherds, R.E.: 
Vertex algebras, Kac-Moody algebras and the monster. 
Proc. Nat. Acad. Sc. \textbf{83}, 3068--3071 (1986).


\bibitem
{Bu} 
Burnside, W.
On a class of automorphic functions, 
Proc.L.Math.Soc.  \textbf{23} 49--88 (1891).





\bibitem
{Co} 
Codogni, G.
Vertex algebras and Teichmuller modular forms, 
arXiv:1901.03079.




\bibitem
{DFK} 
Ph. Di Francesco and R. Kedem, Q-systems as cluster algebras. II.
 {\it Lett. Math. Phys.} {\bf 89}, no. 3, (2009) 183. 




\bibitem
{DGM} L. Dolan, P. Goddard and P. Montague,  Conformal field theories,
  representations and lattice constructions,
{\it  Commun.~Math.~Phys.} \textit{179} (1996) p. 61. 


 \bibitem
{DL} C. Dong and J. Lepowsky, \textit{Generalized Vertex Algebras and Relative
 Vertex Operators} (Progress~in~Math. Vol. \textbf{112}, Birkh\"{a}user, Boston, 1993).




\bibitem
{EO}
Eguchi, T. and Ooguri, H.
Conformal and current algebras on a general Riemann surface,
Nucl. Phys. \textbf{B282} 308--328 (1987).

\bibitem
{Fa} 
Fay, J.D. 
\textit{Theta Functions on Riemann Surfaces},
Lecture Notes in Mathematics, 
Vol. 352. Springer-Verlag, (Berlin-New York, 1973). 





\bibitem{FMS} Ph. Di Francesco,~P., Mathieu,~P. and Senechal,~D.:
Conformal Field Theory.
Springer Graduate Texts in Contemporary Physics, Springer-Verlag, New York (1997).


 \bibitem
{FK}  H.M. Farkas and  I. Kra, \textit{Riemann surfaces}, 
(Springer-Verlag, New York, 1980).




\bibitem
{FST} S. Fomin, M. Shapiro, D. Thurston, Cluster algebras and 
 {\it Acta Mathematica} {\bf 201} (2008) 83. 


\bibitem
{FZ1} S. Fomin, A. Zelevinsky, Cluster algebras. I. Foundations,
 {\it J. of the Amer. Math. Soc.} (2002) {\bf 15} (2) 49. 7--529 . 

\bibitem
{FZ2} S. Fomin, A. Zelevinsky. Cluster algebras. II. Finite type classification,  
{\it Inventiones Mathematicae} {\bf 154} (1): (2003) 63--121 (2003). 

\bibitem
{FZ3} S. Fomin, A. Zelevinsky. Cluster algebras. IV. Coefficients,  
{\it Compositio Mathematica} {\bf 143} (1): (2007) p. 112--164 (2007). 




\bibitem
{FG1} V. Fock and A. Goncharov, Moduli spaces of local systems and higher Teichm\"uller theory. 
{\it Publ. Math. Inst. Hautes \'Etudes Sci}. (2006) {\bf 103}, 1.

\bibitem
{FG2} V. Fock and A. Goncharov, 
 Cluster ensembles, quantization and the dilogarithm, 
{\it Ann. Sci. \'Ec. Norm. Sup\'er.} (4) {\bf 42}, no. 6, (2009) 865.

\bibitem
{FG3} V. Fock and A. Goncharov, 
 Cluster ensembles, quantization and the dilogarithm. II, 
{\it in  The intertwiner. Algebra, arithmetic, and geometry: in honor of Yu. I. Manin.} Vol. I, 
(Progr. Math., 269, Birkh\"auser Boston, Inc., Boston, MA, 2009) p. 655. 

\bibitem
{FG4} V. Fock and A. Goncharov, Dual Teichm\"uller and lamination spaces, 
{\it in Handbook of Teichm\"uller theory.} Vol. I, 
IRMA Lect. Math. Theor. Phys., 11 (Eur. Math. Soc., Z\"urich, 2007), p. 647.  




\bibitem
{FHL} I. Frenkel,Y.-Z. Huang and J. Lepowsky, 
\textit{On Axiomatic Approaches to Vertex Operator Algebras and Modules}, 
{\it Mem.~AMS.} \textbf{104} No. 494 (1993).




\bibitem
{FLM} I. Frenkel, J. Lepowsky and A. Meurman, 
\textit{Vertex Operator Algebras and the Monster}, {\it Pure and Appl. Math.} Vol. 
\textbf{134} (Academic Press, Boston, 1988).




\bibitem{FS}  Freidan, D. and Shenker, S.:
The analytic geometry of two-dimensional conformal field theory. 
Nucl.~Phys. \textbf{B281} 509--545 (1987).



\bibitem
{Fo} Ford, L.R.
\textit{Automorphic Functions},
AMS-Chelsea, (Providence, 2004).

%
\bibitem
{GSV1} M. Gekhtman, M. Shapiro and A. Vainshtein, 
{\it Mosc. Math. J.} {\bf 3}, no. 3, (2003) 899. 

\bibitem
{GSV2} M. Gekhtman, M. Shapiro and A. Vainshtein, 
{\it Duke Math. J.} {\bf 127}, no. 2, (2005) 291., 





\bibitem
{GSV3} M. Gekhtman, M. Shapiro, and A. Vainshtein. Cluster algebras
and Poisson geometry,  
{\it in Mathematical Surveys and Monographs,} vol. 167 
(American Mathematical Society, Providence, RI, 2010).

\bibitem
{GLS} Ch. Geiss, B. Leclerc and J. Schr\"oer. Cluster structures on 
{\it Selecta Math.} (N.S.) 19, no. 2, (2013) 337. 



\bibitem
{HL} D.Hernandez and B. Leclerc. Cluster algebras and quantum affine algebras, 
{\it Duke Math. J.} 154, no. 2, (2010) 265. 




\bibitem
{K} V. Kac, \textit{Vertex Algebras for Beginners, Second Ed.} (Univ. Lect. Ser. \textbf{10}, AMS, 1998).
 

\bibitem
{Ke1} B. Keller,  Cluster algebras, quiver representations and triangulated 
 categories, 
 Triangulated categories", 76--160,  
{\it in London Math. Soc. Lecture Note Ser.,}  (Cambridge Univ. Press, Cambridge, 2010) p. 375. 

\bibitem
{Ke2} B. Keller,
 Quantum loop algebras, quiver varieties, and cluster algebras, 
{\it in 
Representations of algebras and related topics,} 
 (EMS Ser. Congr. Rep., Eur. Math. Soc., Z\"urich, 2011) p. 117. 



\bibitem
{KS} M. Kontsevich and Y. Soibelman, Stability structures, Donaldson-
Thomas invariants and cluster transformations, arXiv:0811.2435, 2008. 



\bibitem{Kn}  Knizhnik, V.G.:
Multiloop amplitudes in the theory of quantum strings and complex geometry.
 Sov.~Phys.~Usp. \textbf{32} 945--971 (1989).


\bibitem
{LL} J. Lepowsky and H. Li, \textit{Introduction to Vertex Algebras}, (Progress
 in Math. Vol. \textbf{227}, Birkh\"auser, Boston, 2004).


\bibitem
{L} S. Lang, \textit{Introduction to Modular Forms} (Springer, Berlin, 1976).



\bibitem{La} Lang, S.:
Elliptic functions. 
Springer-Verlag, New York (1987).


\bibitem{Li} Li, H.: 
Symmetric invariant bilinear forms on vertex operator algebras. 
J. Pure. Appl. Alg. \textbf{96}, 279--297 (1994).




\bibitem[McI]{McI} McIntyre, A. 
Analytic torsion and Faddeev-Popov ghosts,
SUNY PhD thesis 2002, hdl.handle.net/11209/10688.



\bibitem
{McIT}
McIntyre, A. and  Takhtajan, L.A. 
Holomorphic factorization of determinants of Laplacians on Riemann surfaces and a higher genus generalization of Kronecker's first limit formula,
GAFA, Geom.Funct.Anal. \textbf{16}  1291--1323 (2006).



\bibitem
{MN} A. Matsuo and K. Nagatomo, 
\textit{Axioms for a Vertex Algebra and the Locality of Quantum Fields},   
(Math.~Soc.~of~Japan~Memoirs Vol. \textbf{4}, Tokyo, 1999).

\bibitem
{Miy1} M. Miyamoto, Modular invariance of vertex operator algebras satisfying $C_2$-cofiniteness, 
{\it Duke~Math.~J.~Vol.} \textbf{122} No. 1 (2004) p. 51. 



\bibitem
{Mu} 
Mumford, D.
\textit{Tata Lectures on Theta I and II},
Birkh\"{a}user, (Boston, 1983).

\bibitem
{O} 
Odesskii, A.
Deformations of complex structures on Riemann surfaces and integrable structures of Whitham type
hierarchies,
arXiv:1505.07779.

\bibitem
{R}
Rauch, H.E. 
On the  transcendental moduli of algebraic Riemann surfaces,
Proc. Nat. Acad. Sc. \textbf{11} 42--48 (1955).




\bibitem
{N1} T. Nakanishi, Dilogarithm identities for conformal field theories and cluster algebras: simply laced case,
{\it Nagoya Math. J.} {\bf 202}, (2011) 23--43 .

\bibitem
{N2} T. Nakanishi, Periodicities in cluster algebras and dilogarithm identities, 
{\it in 
Representations of algebras and related topics,} 
(EMS Ser. Congr. Rep., Eur. Math. Soc., Z\"urich, 2011) p. 407.  


\bibitem
{Na} 
H. Nakajima, Quiver varieties and cluster algebras, 
 {\it Kyoto J. Math.}{\bf 51} no. 1, (2011) p. 71--126 (2011). 








\bibitem
{Se} J.-P. Serre, \textit{A Course in Arithmetic} (Springer, New York, 1973). 
 


\bibitem
{Sch}  R. Schiffler, 
 On cluster algebras arising from unpunctured surfaces. II. 
{\it Adv. Math.} {\bf 223} no. 6 (2010) 1885. 


\bibitem
{TW} M. P. Tuite, M. Welby. General Genus Zhu Recursion for Vertex Operator Algebras, arXiv:1911.06596. 



\bibitem{Y} Yamada, A.:
Precise variational formulas for abelian  differentials. 
Kodai. Math. J. \textbf{3}, 114--143 (1980).



\bibitem
{Z}
Zhu, Y.
Modular-invariance of characters of vertex operator algebras,
J.Amer.Math.Soc. \textbf{9} 237--302  (1996).




 
\end{thebibliography}
\end{document}